\documentclass[11pt,amssymb]{amsart}
\usepackage{amssymb}
\usepackage[english]{babel}
\usepackage[mathscr]{eucal}
\usepackage[all,cmtip]{xy}
\usepackage[utf8]{inputenc}

\usepackage{mathtools}
\DeclarePairedDelimiter{\norm}{\lVert}{\rVert}

\newtheorem{theo}{Theorem}[section]

\newtheorem{lemma}[theo]{Lemma}
\newtheorem{prop}[theo]{Proposition}
\newtheorem{coro}[theo]{Corollary}
\newtheorem{rema}[theo]{Remark}
\newtheorem{exam}[theo]{Example}
\newtheorem{thm}[theo]{Theorem}

\newtheorem{defi}[theo]{Definition}

.240pk scaled 1200 .240pk
\begin{document}
\title[Bounded Generation by semi-simple elements: quantitative results]{Bounded Generation by semi-simple elements: quantitative results}

\author[P.~Corvaja]{Pietro Corvaja}
\author[J.~Demeio]{Julian L. Demeio}
\author[A.~Rapinchuk]{Andrei S. Rapinchuk}
\author[J.~Ren]{Jinbo Ren}
\author[U.~Zannier]{Umberto M. Zannier}

\address{Dipartimento di Scienze Matematiche, Informatiche e Fisiche, via delle Scienze,
206, 33100 Udine, Italy}

\email{pietro.corvaja@uniud.it}

\address{Max Planck Institute for Mathematics, Bonn, 53111, Germany}

\email{demeio@mpim-bonn.mpg.de}

\address{Department of Mathematics, University of Virginia,
Charlottesville, VA 22904-4137, USA}

\email{asr3x@virginia.edu}

\address{School of Mathematics, Institute for Advanced Study,
Princeton, NJ 08540, USA}

\email{jren@ias.edu}

\address{Scuola Normale Superiore, Piazza dei Cavalieri, 7, 56126 Pisa, Italy}

\email{umberto.zannier@sns.it}


\begin{abstract}
We prove that for a number field $F$, the distribution of the points of a set $\Sigma \subset \mathbb{A}_F^n$ with a purely exponential parametrization, for example a set of matrices boundedly generated by semi-simple (diagonalizable) elements, is of at most logarithmic size when ordered by height. As a consequence, one obtains that a linear group $\Gamma\subset \mathrm{GL}_n(K)$ over a field $K$ of characteristic zero admits a purely exponential parametrization if and only if it is finitely generated and the connected component of its Zariski closure is a torus. Our results are obtained via a key inequality about the heights of minimal $m$-tuples for purely exponential parametrizations. One main ingredient of our proof is Evertse's strengthening of the $S$-Unit Equation Theorem. 

\end{abstract}

\maketitle

\section{Statement of Main Results}
The purpose of this note is to annouce and sketch certain results in a future paper by the current authors \cite{CDRRZ}.

We start by the notion of bounded generation. An abstract group $\Gamma$ is said to have the  {\bf bounded generation} property (BG) if it can be written in the form
$$\Gamma=\langle \gamma_1 \rangle \cdots \langle \gamma_r \rangle$$
for certain fixed $\gamma_1,\dots, \gamma_r\in \Gamma$, where $\langle \gamma_i \rangle$ is the cyclic subgroup generated by $\gamma_i$. We refer the interested readers to the discussion in Section 1 of \cite{CRRZ} and the references therein for the motivation for
(BG). In \cite{CRRZ}, it was shown that a linear group $\Gamma \subset \mathrm{GL}_n(K)$ over a field $K$ of characteristic zero   ``usually lacks
   (BG) by semi-simple elements'', i.e. (BG) such that all $\gamma_i$ are diagonalizable. More precisely, it was shown in \cite{CRRZ} that if a linear group $\Gamma$ over a field of characteristic zero consists entirely of semi-simple elements, then $\Gamma$ has (BG) if and only if it is finitely generated and virtually abelian. In particular, if $K$ is a number field and $S$ is a finite set of places including all infinite ones, then \underline{infinite} $S$-arithmetic subgroups of absolutely almost simple $K$-anisotropic groups never have (BG).

The current paper will significantly strengthen the above results by providing some quantitative properties which describe the extent of the absence of (BG) by semi-simple elements. In fact, we will consider the following more flexible question in terms of purely exponential polynomial parametrizations (PEP). 

\vskip2mm

\noindent {\bf Definition.} {\it Let $\Sigma $ be a subset of a variety $V\subset \mathbb{A}_K^n$ ($K$ is a field). Then $\Sigma$ is said to have {\bf Purely Exponential Parametrization} (PEP) in $r$ variables if $\Sigma$ has shape
$$\Sigma=\Big\{ (f_1(\mathbf{n}),\dots,f_s(\mathbf{n}));\mathbf{n}\in \mathbb{Z}^r \Big\},$$
where each $f_i(\mathbf{x})=f(x_1,\dots, x_r)$ is a {\bf Purely Exponential Polynomial}, i.e. an expression of the form
$$f_i(\mathbf{x})=\sum_{j=1}^e a_j \lambda_1^{l_{1,j}(\mathbf{x})}\cdots \lambda_k^{l_{k,j}(\mathbf{x})},$$
for certain constants $a_1,\dots,a_e,\lambda_1,\dots, \lambda_k\in \overline{K}^{\times}$ and linear forms $l_{j,s}(\mathbf{x})$ in $r$ variables whose coefficients are {\bf rational integers}. Here we refer to the elements $\lambda_1,\dots,\lambda_k$ as the {\bf bases} of $\mathbf{f}\colon =(f_1,\dots,f_s)$, to the linear forms $l_{i,j}$ as the {\bf exponents} of $\mathbf{f}$, and to the constants $a_j$ as the {\bf coefficients} of $\mathbf{f}$.}
\begin{rema}
In the definition as above, we do not require that all coefficiens and bases are in $K$. Also, it is easy to see that any finite union of (PEP) sets is still a (PEP) set.
\end{rema}

\begin{exam} The classical Pell equations naturally produce (PEP) sets.  For example, the set of integer solutions of $x^2-2y^2=1$, which corresponds to the integer points of the special orthogonal group for the quadratic form $h=x^2-2y^2$, is given by
$$
\left\{   \left( (-1)^m \left(\frac{(3-2\sqrt{2})^n+(3+2\sqrt{2})^n}{2}    \right), \left( \frac{(3-2\sqrt{2})^n-(3+2\sqrt{2})^n}{2\sqrt{2}}    \right)  \right) ; m,n\in \mathbb{Z} \right\}.
$$
\end{exam}

\begin{exam}
Linear groups $\Gamma$ admitting (BG) by semi-simple elements, which are main study objects of \cite{CRRZ}, become typical examples of (PEP) sets. In fact, if $\Sigma=\Gamma\subset \mathrm{GL}_n(K)$ with $\Gamma=\langle \gamma_1\rangle\cdots \langle \gamma_r \rangle$ with the $\gamma_i$'s semi-simple, then there exist $g_i\in \mathrm{GL}_n(\overline{K})$ and $\lambda_{i,j}$ for $i=1,\dots, r,j=1,\dots,n$ with
$$g_i^{-1}\gamma_i g_i=\mathrm{diag}(\lambda_{i,1},\dots, \lambda_{i,n}),\text{ for all }i=1,\dots,r.$$
This implies that every $\gamma \in \Gamma$ has shape
$$\gamma=\prod_{i=1}^r g_i\left[\mathrm{diag}(\lambda_{i,1}^{a_i},\dots, \lambda_{i,n}^{a_i})\right]g_i^{-1}\text{ for some }a_1,\dots,a_r\in \mathbb{Z}.$$
Comparing entries of the two sides of the above relation, we realize $\Sigma$ as a (PEP) set $\subset \mathbb{A}_K^{n^2}$ in $r$ variables with bases equal to those eigenvlues $\lambda_{i,j}$'s.
\end{exam}

In the current article, we will provide some sparseness results for (PEP) subsets of affine varieties $V\subset \mathbb{A}_K^n$ over a \underline{number field} $K$. The language we are using to describe sparseness is the {\bf height function} on the affine space $K^n$, defined by 
$$H_{\text{aff}}(x_1,\dots,x_n)\colon =H(1:x_1:\dots:x_n)\colon =\Big( \prod_{v \in V_K}\max\{ 1,\| x_1 \|_v,\dots, \| x_n \|_v  \}\Big)^{1\slash [K\colon \mathbb{Q}]}$$
where $V_K$ is the set of all places of $K$, and $\|\cdot \|_v$ are normalized $v$-adic valuations such that the product formula holds. We will also use the corresponding logarithmic height $h_{\text{aff}}\colon = \log H_{\text{aff}}$. See \cite[\S B]{HindrySilverman} or \cite{BombieriGubler} for details about height functions. The first main result of this paper, which is about the distribution of (PEP) sets, can be stated as follows.
\begin{theo}[First Main Theorem: quantitative result] \label{firstmainthm}Let $\mathbb{A}_K^n$ be an affine space over a number field $K$, then for any (PEP) set $\Sigma \subset \mathbb{A}_K^n$ in $r$ variables, we have
$$\Big| \{P\in \Sigma;H_{\text{aff}}(P)\leq H\}\Big|=O((\log H)^r)\text{ when }H\to \infty.$$
In other words, any (PEP) set has at most {\bf logarithmic-to-the-}$r$ growth in terms of the height. 
\end{theo}
\begin{rema} In order to interprete Theorem \ref{firstmainthm} as a \underline{sparseness} result, it should be emphasized that there is a highly involved but also well-developed topic about ``counting lattice points in Lie groups''. In particular, \cite[Corollary 1.1]{Maucourant07} (see also \cite[Theorem 2.7 and Theorem 7.4]{GW07}) informs us that points in any lattice of a non-compact semi-simple Lie group $\mathcal{G}$ with finite center have growth rate $cH^d(\log H)^e, H\to \infty$ for certain $c,d>0,e\geq 0$ in terms of an Euclidean norm on $\mathbb{R}^{n^2}\supset \mathrm{GL}_n(\mathbb{R}) \supset \mathcal{G}$. As a consequence, we see that for a semi-simple algebraic group $G\subset \mathrm{GL}_n$ over $\mathbb{Q}$ of non-compact type,  Theorem \ref{firstmainthm} provides sparseness, in terms of the height, for all (PEP) subsets of $\Gamma \colon =G(\mathbb{Z})\colon =\mathrm{GL}_n(\mathbb{Z})\cap G(\mathbb{R})$. As a more explicit example, according to \cite[Example 1.6]{DukeRudnickSarnak}, the set $\left\{\mathbf{s}\in \mathrm{SL}_n(\mathbb{Z});H_{\mathrm{aff}}(\mathbf{s})\leq H\right\}$ is of order $cH^{n^2-n}$ for some $c>0$, therefore any (PEP) set in $\Gamma=\mathrm{SL}_n(\mathbb{Z})$, which has only logarithmic growth, is sparse in terms of the height. Verification of sparseness for (PEP) subsets of many other $S$-arithmetic groups, following strategies developed in \cite{GW07} and \cite{GN12}, will be available in \cite{CDRRZ}.

\end{rema}

If we apply Theorem \ref{firstmainthm} to the particular situation of (BG) by semi-simple elements, we acquire the following consequence. 

\begin{coro}\label{corsparsebgsemisimple} Let $\Gamma \subset \mathrm{GL}_n(K)$ be a linear group over a number field $K$, then for any semi-simple elements $\gamma_1,\dots, \gamma_r\in \Gamma$, we have
$$\Big| \{P\in \langle \gamma_1\rangle\cdots \langle \gamma_r\rangle ;H_{\text{aff}}(P)\leq H\}\Big|=O((\log H)^r)\text{ when }H\to \infty.$$
\end{coro}

The proof of Theorem \ref{firstmainthm} relies crucially on a key statement about the so-called ``minimal $m$-tuples'' with respect to a (PEP) set which seems to be of independent interest. 

\begin{defi} Given a vector ${\mathbf{f}}=(f_1,\dots, f_s)$ of exponential polynomials in $r$ variables, i.e. each $f_j$ is an exponential polynomial in $r$ variables, an element ${\mathbf{n}}=(n_1,\dots,n_r)\in \mathbb{Z}^r$ is called $\mathbf{f}$-{\bf minimal} (or {\bf minimal} with respect to $\mathbf{f}$) if for all ${\mathbf{n'}}=(n_1',\dots, n_r')\in \mathbb{Z}^m$ with ${\mathbf{f}}({\mathbf{n'}})={\mathbf{f}}({\mathbf{n}})$ (i.e. $f_j({\mathbf{n'}})=f_j({\mathbf{n}})$ for all $j$), we have $\| \mathbf{n'} \|_{\infty} \colon =\max\{|n_1'|,\dots, |n_r'|\}\geq \max\{|n_1|,\dots, |n_r|\}=\colon \| \mathbf{n} \|_{\infty} $.
\end{defi} 

\begin{theo}[Primary Height Inequality] \label{minspecial} Let $\mathbf{f}$ be a vector of purely exponential polynomials in $r$ variables, then there exists a constant $C=C(\mathbf{f})>0$ such that for all $\mathbf{f}$-minimal vectors $\mathbf{n}\in \mathbf{Z}^r$, we have

\begin{equation}
h_{\mathrm{aff}}({\mathbf{f}}({\mathbf{n}}))\geq C\cdot \|\mathbf{n}\|_{\infty}
\end{equation}
except on some set of the form $\mathbf{f}^{-1}(A)$ with $A$ finite. 
\end{theo}
It should be emphasized that the constant $C$ above will be explicitly computable, while the cardinality of the set $A$ in Theorem \ref{minspecial} is non-effective in general, see Remark \ref{effective}.

The first main Theorem, i.e. Theorem \ref{firstmainthm}, being quantitative itself, leads us to the following qualitative theorem which fully describes all linear groups admitting (BG) by semi-simple elements (or (PEP)). It is worth pointing out that, thanks to a specialization argument, the following result works for linear groups over {\bf arbitrary} fields of characteristic zero.
\begin{theo}[Second Main Theorem: qualitative result]\label{secondmainthm} 
Let $K$ be a field of characteristic zero and let $\Gamma \subset \mathrm{GL}_n(K)$ be a linear group. Then the following three properties are equivalent. 
\begin{enumerate}
\item $\Gamma$ has (PEP).
\item $\Gamma$ consists only of semi-simple elements and has (BG).
\item $\Gamma$ is finitely generated and the connected component $G^{\circ}$ of the Zariski closure $G$ of $\Gamma$ is a torus (in particular, $\Gamma$ is virtually abelian). 
\end{enumerate}
\end{theo}

This result serves as an extension of one main Theorem in \cite[Theorem 1.1]{CRRZ} which claims that if a linear group over a field of characteristic zero has (BG) by semi-simple elements, then it is virtually solvable. More importantly, Theorem \ref{secondmainthm} gives a {\bf complete} answer to the Questions asked in \cite[p. 3]{CRRZ}. 

\section{Brief outline of proofs}
It is straightforward to verify that Theorem \ref{minspecial} implies Theorem \ref{firstmainthm}. For simplicity of
argument, we only sketch the proof of Theorem \ref{minspecial} for $\mathbf{f}=f$ being a single purely exponential polynomial.
 The sketch we give here follows the lines of the proof in the general case, and already includes all the main ideas and ingredients in the counterpart in \cite{CDRRZ}.

\begin{proof}[Sketch of proof of Theorem \ref{minspecial}]
	The proof goes by induction on $r$, the number of variables in $\mathbf{n}$. The base case when $r=0$ is trivial, now let $r\geq 1$.
	We write:
	\[
	f(\mathbf{n})=\sum_{i=1}^e  a_i  u_i(\mathbf{n}),
	\]
	where $\lambda_j, a_i\in K^*$ and $u_i(\mathbf{n})=\lambda_{1}^{l_{1,i}(n_1,\ldots,n_r)} \cdots \lambda_{k}^{l_{k,i}(n_1,\ldots,n_r)}$ are purely exponential monomials. 
	
	Some non-trivial but routine manipulations enable one to reduce to the case where $\lambda_1,\ldots,\lambda_k$ are multiplicatevely independent, i.e. $\lambda_1^{\theta_1}\cdots \lambda_k^{\theta_k}=1(\theta_j\in \mathbb{Z}) \Longleftrightarrow \theta_1=\cdots \theta_k=0$, and 
	where the linear forms 
	$l_{i,j}$ span the dual space of $\mathbb{Q}^{r}$ over $\mathbb{Q}$.

	We need the following crucial height inequality which can be derived from a result of Evertse \cite[Theorem 6.1.1]{EG15}  (which is itself a consequence of the Schlickewei-Schmidt Subspace Theorem, cf. \cite[Theorem 2.2]{CorvajaZannier}).
	\begin{thm}[Evertse] \label{evertseinequality}
		Let $S$ be a finite set of places of a number field $K$ containing all archimedean ones. Then there exists an effective $C>0$ such that the inequality
		\[
		h_{\mathrm{aff}}(s_1+\cdots+s_e) < C \cdot (h_{\mathrm{aff}}(s_1)+\cdots+h_{\mathrm{aff}}(s_e))\text{   with   }s_i\in \mathcal{O}_S^{\times}
		\]
		has only finitely many solutions such that the sum $s_1+\cdots+s_e$ is non-degenerate.
	\end{thm}
       Here {\bf non-degenerate} refers to the fact that $\sum_{i \in I} s_i \neq 0$ for any nonempty proper subset $I\subseteq \{1,\ldots,e\}$.
	
	Using Theorem \ref{evertseinequality} by taking the set $S$ of places such that all bases $\lambda_i$ and coefficients $a_j$ of $f$ are $S$-units, we obtain that for certain $C'>0$, the inequality:
	\begin{equation}\label{Eq:inequality}
		h_{\mathrm{aff}}(f(\mathbf{n})) \geq C'\cdot (h_{\mathrm{aff}}(u_1(\mathbf{n}))+ \ldots + h_{\mathrm{aff}}(u_e(\mathbf{n})))
	\end{equation}
	holds for all but finitely many $\mathbf{n} \in \mathbb{Z}^r$ such that the sum defining $f(\mathbf{n})$ is non-degenerate. 
	
	Recall the following standard fact \cite[p.118, Eq. (3.12)]{Zannier09}:
	\begin{prop}
		Let $m\in \mathbb{N}$ and $\phi =(\phi_1,\ldots,\phi_e):\mathbb{Z}^r \to (K^*)^e$ be an injective group homomorphism. Then there are constants $C_2>C_1>0$ such that for every $\mathbf{n} \in \mathbb{Z}^r$ the following inequalities hold:
		\[
		C_1 \norm{\mathbf{n}}_{\infty} \leq h_{\mathrm{aff}}(\phi_1(\mathbf{n}))+\ldots+h_{\mathrm{aff}}(\phi_e(\mathbf{n})) \leq C_2 \norm{\mathbf{n}}_{\infty}.
		\]
	\end{prop}
	
	Using the proposition above with $\phi=(u_1,\ldots,u_e)$, which is injective because of the assumption that the linear forms $l_{i,j}$ span $(\mathbb{Q}^r)^{\vee}$ and that those $\lambda_j$'s are multiplicatively independent, one deduces that the right hand side of \eqref{Eq:inequality} is $\asymp \norm{n}_{\infty}$. This completes the argument for the non-degenerate case. 
	
	Now consider those $\mathbf{n} \in \mathbb{Z}^r$ such that the sum defining $f(\mathbf{n})$ is degenerate. Then we may take a proper subset $I=\{i_1,\dots,i_t\} \subset \{1,\dots, e\} (t<e)$ with $a_{i_1}u_{i_1}(\mathbf{n})+\cdots+a_{i_t}u_{i_t}(\mathbf{n})=0$.
	
	We are now in a position to use Laurent's theorem \cite[Theorem 10.10.1]{EG15}, which can also be  deduced from Theorem \ref{evertseinequality}.
	\begin{theo}[Laurent]\label{Laurent}
		Let $K$ be a number field, $\Gamma \subseteq (K^*)^t$ be a finitely generated subgroup, and let $X$ be a subvariety of $(\mathbb{G}_m)^t$. Then the Zariski closure of $\Gamma \cap X$ is a finite union of cosets of $(\mathbb{G}_m)^t$.
	\end{theo}
	
	Employing
	Laurent's theorem on
	the subgroup $\Gamma=\mathrm{im} (\phi=(u_{i_1},\ldots,u_{i_t}):\mathbb{Z}^r \to (K^*)^t)$ and the hyperplane $X: a_{i_1}x_{i_1}+\ldots+a_{i_t}x_{i_t}=0$, and letting
	$I$ go through all proper subsets of $\{1,\dots,e\}$ (finitely many possibilities), we deduce that the set of such $\mathbf{n}$ is contained in a finite union of cosets of $\mathbb{Z}^r$. Moreover, due to the assumption that the linear forms $l_{i,j}$ span $(\mathbb{Q}^r)^{\vee}$, we may assume these cosets are all translates of subgroups of rank $<r$.

       Taking the restriction of $f$ to each of the above proper cosets, and composing it with a suitable affine transformation, we produce finitely many (PEP) sets whose parametrizations all involve \underline{$<r$ variables}. Applying the induction hypothesis to these new (PEP) sets, the proof is completed. 
\end{proof}

\begin{rema}[effectiveness]\label{effective}
	Taking more care in the proof above, one can actually make the constant $C$ in Theorem \ref{minspecial} to be effectively computable in terms of  $\mathbf{f}$. 
	
	However, our approach can say little about the effectiveness of the exceptional set $A$ (and even less about the effectiveness of $f^{-1}(A)$) of Theorem \ref{minspecial}, not even its cardinality. As a consequence, in the context of Theorem \ref{firstmainthm}, we are unable to explicitly compute a constant $a>0$ such that ${\Big| \{P\in \Sigma;H_{\text{aff}}(P)\leq H\}\Big|<a \cdot (\log H)^r}$ for sufficiently large $H$.

	This is in sharp contrast with the situation of 
	$S$-unit equations, e.g. $x_1+\cdots+x_s=1, x_i\in \mathcal{O}_S^*$, whose {\bf number} of non-degenerate solutions can be effectively boundable from above, cf. the seminal paper \cite{effective02} and its refinement \cite{effective09}, see also \cite{effectiveremond} for another approach.  
	
         In fact, we prove in \cite{CDRRZ} that an effective bound for
         the cardinality of $A$ in Theorem \ref{minspecial} would yield an explicit bound for the {\bf size} of non-degenerate solutions to an arbitrary $S$-unit equation, which is still an open problem. Thus, the non-effectiveness of the exceptional set $A$ of Theorem \ref{minspecial} lies deeply in the openness of the difficult effectiveness problem of the Schlickewei-Schmidt subspace theorem.
\end{rema}

We now turn to the discussion of the second main result, Theorem \ref{secondmainthm}. The proof of Theorem \ref{secondmainthm}, being non-trivial though, is roughly analogous to that of Theorem 1.1 and Corollary 1.2 of \cite{CRRZ}. In particular, the theory of generic elements, cf. \cite{PrR1},\cite{PrR2}, \cite{PrR3}, will be needed again.  We will omit the full verification here for simplicity of presentation. In the following we will only highlight two new ingredients in the proof of Theorem \ref{secondmainthm} and postpone detailed arguments in \cite{CDRRZ}.

The first one is a consequence of Theorem \ref{firstmainthm}.

\begin{coro}\label{Cor}
	Let $K$ be a number field, $\Sigma \subseteq \mathrm{GL}_n(K)$ be a (PEP) subset, and let $g\in \mathrm{GL}_n(K)$ be a non-semi-simple matrix.
	Then there is an $m \in \mathbb{N}$ such that $\Big| \{n\in \mathbb{N};n\leq N\text{ and }g^n\in \Sigma\}\Big|
	=O(\log^mN)$ as $N\to \infty$ .
\end{coro}

\begin{proof}
    Write $g=g_ug_s$ for the Jordan decomposition of $g$ with $g_u$ unipotent, $g_s$ semisimple and $[g_u,g_s]=1$. Note that the condition $g^n=(g_ug_s)^n \in \Sigma$ implies that $g_u^n \in \Sigma \cdot\langle  g_s\rangle$, and that the subset $\Sigma'= \Sigma\cdot \langle g_s\rangle \subseteq \mathrm{GL}_n(K)$ is also a (PEP) set. So, we reduce to proving the result for $g_u$. We may, therefore, assume that $g$ is unipotent.
   
	By writing $g=id+g_N$ with $g_N$ nilpotent, and considering the binomial expansion of ${g^n=(id+g_N)^n}$, it is easy to check that the height of the coefficients of $g^n$ has polynomial growth in $n$. Due to Theorem \ref{firstmainthm}, the elements of height $\leq H$ in the (PEP) set $\Sigma$ grow at most as some power of $\log H$ as $H \to \infty$. This proves the corollary. 
\end{proof}

The following second requires a not entirely trivial argument which uses the finiteness of non-degenerate solutions to $S$-unit equations (cf. \cite{CDRRZ}).

\begin{lemma}
	Let $f:\mathbb{Z}^r \to K^*$ be a (PEP). If its image is a multiplicative subgroup of $K^*$, then this subgroup is finitely generated.
\end{lemma}

Details of the proofs in this section as well as relevant examples and remarks will appear in \cite{CDRRZ}.

\vskip2mm

\noindent {{\small {\bf Acknowledgements.} The first author is partially funded by the Italian PRIN 2017 ``Geometric, algebraic and analytic methods in arithmetic''. The second author was a guest at the Max Planck Institute for Mathematics when working on this article. He thanks the Institute for their hospitality and their financial support. The fourth author is supported by the Institute for Advanced Study and the National Science Foundation under Grant No. DMS-1926686.}}

\bibliographystyle{amsplain}

\end{document}